\documentclass[11pt,a4paper]{amsart}
\usepackage[utf8]{inputenc}
\usepackage{enumerate}
\usepackage{amsmath,amssymb,mathtools,amsthm}
\usepackage[dvips]{graphicx}
\usepackage[font=footnotesize,skip=1pt]{caption}

\usepackage{lmodern}
\usepackage[T1]{fontenc}
\usepackage{newunicodechar}
\newunicodechar{̆}{\u{}}

\usepackage{comment}
\usepackage{setspace}
\usepackage[dvipsnames]{xcolor}
\usepackage{amsmath, amssymb, amsfonts, graphics, graphicx, float}
\usepackage{multirow}
\usepackage{bigdelim}
\usepackage{parskip}
\usepackage{enumerate}
\usepackage{csquotes} 
\usepackage{pgf,tikz}
\usepackage{listings}

\usepackage{wasysym}
\usepackage{amssymb,amsfonts, color,epsf}
\usepackage{verbatim}
\usepackage{fullpage}
\usepackage{ytableau}
\usepackage{young}
\usepackage{youngtab}
\usepackage[left=30mm, right=30mm, bottom=30mm, top=30mm]{geometry}
\ifx\pdfpageheight\undefined

   \usepackage[dvips]{graphicx}
   
   \makeatletter
   \edef\Gin@extensions{\Gin@extensions,.mps}
   \makeatother
\else

  \usepackage[breaklinks,bookmarks=false, hidelinks,colorlinks,citecolor=red,urlcolor=blue,hypertexnames=false]{hyperref}
  \usepackage{cleveref}
\fi
\usepackage{bm}

\usepackage{amsfonts,mathrsfs}
\usepackage{bbm,amsmath,amssymb,amsfonts,graphicx,color,subfig,mathtools, verbatim}

\usepackage{amsmath}
\usepackage{tabularx}
\usepackage{color, colortbl}
\usepackage{url}
\makeindex

\numberwithin{equation}{section}
\newtheorem{thm}{Theorem}
\newtheorem{prop}[thm]{Proposition}
\newtheorem{lemma}[thm]{Lemma}
\newtheorem{cor}[thm]{Corollary}

\newtheorem{definition}[thm]{Definition}
\newtheorem{remark}[thm]{Remark}
\theoremstyle{definition}

\definecolor{RED}{rgb}{0.6,0,0}
\newcommand{\N}{\mathbb{N}}

\newcommand{\R}{\mathbb{R}}
\newcommand{\C}{\mathbb{C}}
\newcommand{\K}{\mathbb{K}}

\newcommand{\xx}{\mathbf{x}}

\DeclareMathOperator{\supp}{supp}
\DeclareMathOperator{\rank}{rank}

\numberwithin{thm}{section}
\usepackage[style=numeric,maxbibnames=99]{biblatex}
\makeatletter
\def\paragraph{\@startsection{paragraph}{4}%
  \z@{.5\linespacing\@plus.7\linespacing}{-.5em}%
  {\normalfont\normalsize\bfseries}}
\makeatother
\newtheoremstyle{break}  
  {3pt}   
  {5pt}   
  {\normalfont}  
  {0pt}       
  {\scshape} 
  {}         
  {4pt}  
  {}          
\theoremstyle{break}

\numberwithin{subcase}{case}
\title{Quadrature rules with few nodes supported on algebraic curves}
\author{Cordian Riener}

\author{Ettore Teixeira Turatti}

\address{Department of Mathematics and Statistics, UiT - the Arctic University of Norway, 9037 Troms\o, Norway}
\email{cordian.riener@uit.no}
\email{ettore.t.turatti@uit.no}

\begin{document}

\begin{abstract}
We investigate quadrature rules for measures supported on real algebraic and rational curves, focusing on the \emph{odd-degree} case $2s-1$. Adopting an optimization viewpoint, we minimize suitable penalty functions over the space of quadrature rules of strength $2s-1$, so that optimal solutions yield rules with the minimal number of nodes. For plane algebraic curves of degree $d$, we derive explicit node bounds depending on $d$ and the number of places at infinity, improving results of  Riener--Schweighofer, and Zalar. For rational curves in arbitrary dimension of degree $d$, we further refine these bounds using the geometry of the parametrization and recover the classical Gaussian quadrature bound when $d=1$. Our results reveal a direct link between the algebraic complexity of the supporting curve and the minimal size of quadrature formulas, providing a unified framework that connects real algebraic geometry, polynomial optimization, and moment theory.
\end{abstract}

\maketitle
\section{Introduction} 

Originally associated with computing areas and volumes, the term \emph{quadrature} now primarily refers to the numerical computation of integrals. Specifically, given a {positive  Borel measure} $\mu$ on $\R^n$, its \emph{support} is the smallest closed subset of $\R^n$ whose complement has measure zero. For a measurable function $f$ with $\operatorname{supp}(\mu) \subseteq \operatorname{dom}(f)$ and $\int |f|\, d\mu < \infty$, the goal of numerical quadrature is to approximate the integral $\int f\, d\mu$ efficiently and accurately. Ideally, such approximations should come with error bounds and require only black-box evaluations of $f$, without access to derivatives or other analytic information.  

One classical approach is via \emph{quadrature rules}~\cite{co1}: for a given measure $\mu$, a quadrature rule of \emph{strength} $s$ is a finite set of nodes $N = \{p_1,\ldots,p_\ell\}\subset \R^n$ together with positive weights $w:N\to \R_{>0}$ such that  
\[
\int p\, d\mu = \sum_{x\in N} w(x) p(x)
\]
for all polynomials $p$ of degree at most $s$. The points in $N$ are called the \emph{nodes} of the quadrature rule. While one could in principle allow signed weights, positivity ensures monotonicity of the integral, improves numerical stability, and often yields sharper error estimates~\cite[Conclusion~3.19]{hac}.  

The existence of such rules for arbitrary measures goes back to classical results of Richter~\cite{RS}, Tchakaloff~\cite{tchakaloff1957formules}. Using Carathéodory’s Theorem it follows from these results that a quadrature rule of strength $s$ always exists with at most  
\[
\binom{n+s}{s} = \dim \R[x_1,\ldots,x_n]_{\leq s}
\]
nodes.  

Beyond their numerical significance, quadrature rules are closely linked to the geometry of moment cones: the set of all truncated moment sequences associated with a given measure forms a convex cone whose extreme points correspond to atomic measures, as observed already in the work of Richter, Tchakaloff, and Bayer--Teichmann. Carathéodory’s theorem provides a general upper bound on the number of atoms needed, but in many cases, especially when the support lies on an algebraic set, this bound is far from sharp. Recent advances exploit algebraic structure to refine these bounds, connecting the minimal number of nodes in quadrature formulas to the dimension and geometry of the underlying variety. The present work continues this direction by combining optimization techniques with real algebraic geometry to derive improved bounds for measures supported on curves, particularly rational ones.  

In the univariate case $n=1$, it was already observed by Gauß~\cite{gauss1815methodvs} that one can always find a quadrature rule of strength $2s-1$ with only $s$ nodes, this is the classical Gaussian quadrature formula. A subtle but important distinction arises here: Carathéodory’s theorem concerns the \emph{moment cone} of all truncated moment functionals of degree $\leq 2s$, giving the minimal number of atoms needed to represent any element of this cone. This naturally leads to bounds for quadrature rules of degree $2s$. However, in the \emph{odd} case $2s-1$, central to Gaussian quadrature and our work, the moment cone no longer fully captures the minimal number of nodes, and additional geometric and optimization arguments become necessary. The present article focuses precisely on this odd-degree setting, where sharper bounds can be obtained by exploiting the geometry of the support.  

Earlier work of Schweighofer and the first author~\cite{RS} derived bounds for measures supported on the plane and on plane algebraic curves, while recent work of Zalar~\cite{zalar} improved these bounds for specific families of curves. Building on these results, we revisit the problem from an optimization perspective: among all quadrature rules with at most $N$ nodes, we minimize a suitable penalty function whose optimal solutions yield quadrature rules with a minimal number of nodes.  

The main contributions of this article are given in Theorem \ref{thm1} below. Our results improve on existing bounds in several settings, recover Gaussian-type optimality in the univariate case, and clarify the role of geometric features such as the degree and topology of the supporting curve.\newpage 

\begin{thm}[Main Results]\label{thm1}
Let $\mu$ be a finite Borel measure supported on a real algebraic curve $C$ and let $s \in \mathbb{N}$.
\begin{enumerate}
    \item 
    If $C \subset \R^2$ is a smooth algebraic curve of degree $d$ with at most $t \le d$ places at infinity, then there exists a quadrature rule of strength $2s-1$ for $\mu$ with at most
    \[
    N \;\le\; d s - \left\lceil \frac{d}{2} \right\rceil + 1 + 2t
    \]
    nodes. If $C$ is compact, this gives $N \le d s - \lceil d/2 \rceil + 1$.
    
    \item 
    If $C \subset \R^n$ is a rational curve with polynomial parametrization \[\phi(t) = (\phi_1(t)/\phi_0(t), \ldots, \phi_n(t)/\phi_0(t))\] and $\deg \phi_i \le D$, and $\phi_0$ has at most $p$ distinct real zeros, then there exists a quadrature rule of strength $2s-1$ for $\mu$ with at most
    \[
    N \;\le\; D s - \left\lceil \frac{D}{2} \right\rceil + p + 1
    \]
    nodes, and a quadrature rule of strength $2s$ for $\mu$ with at most $$
    N\leq Ds+p+1
    $$
    nodes.
\end{enumerate}
In particular, for $D=1$ this recovers the classical Gaussian quadrature bound $N=s$ for measures on lines.
\end{thm}

The proofs of these results combine an optimization framework for quadrature rules with tools from real algebraic geometry, Bézout-type arguments, and parametrization techniques for rational curves, part $(1)$ is detailed in Theorem \ref{prop:planecurve} and part $(2)$ in Theorems \ref{thm:quadraturerational} and \ref{thm:quadraturerationaleven}. Before turning to the technical details, we briefly place our results in the context of existing work.

\paragraph{Related Works:}  
The theory of quadrature rules on algebraic sets is closely intertwined with the classical theory of the truncated moment problem. Foundational work by Curto and Fialkow~\cite{CurtoFialkow1991,CurtoFialkow1996,CurtoFialkow2008} introduced moment matrix methods and flat extension techniques that guarantee the existence of finitely atomic representing measures. Their approach also yields detailed results for truncated moment problems on specific algebraic curves such as the unit circle, the parabola, and other conics. 

The question of the minimal number of nodes for general algebraic curves in the plane has been studied by  Riener and Schweighofer~\cite{RS}. Their work recasts the construction of quadrature rules as a conic optimization problem to obtain refined bounds for measures supported on plane algebraic curves.  Subsequent work by Zalar~\cite{zalar} further sharpened these bounds for certain families of curves using moment matrix methods.  

Parallel to these developments, the geometric structure of truncated moment cones and their Carathéodory numbers has been studied extensively, mostly in the even degree case.  Di Dio–Schmüdgen~\cite{diDioSchmudgen2018} established general bounds in the multidimensional setting, while Di Dio–Kummer~\cite{diDioKummer2021} connected Carathéodory numbers to Hilbert functions and analyzed the facial structure of moment cones on algebraic sets. Their results explain the sharpness of Carathéodory bounds in the even-degree case and clarify why additional geometric and optimization arguments are required in the odd-degree setting considered here. Finally, Kummer and Zalar~\cite{kummer2025positivepolynomialstruncatedmoment} recently presented a complete solution to the truncated moment problem for every plane cubic curve.

Further connections between moment cones, Hilbert function methods, and quadrature node bounds have been developed by Baldi, Blekherman, and Sinn~\cite{BaldiBlekhermanSinn2024}. A broader geometric perspective on cubature formulas for spheres, quadrics, and polynomially defined domains is provided by Cools \cite{Cools_1997, COOLS199921}, Cools-Rabinowitz \cite{COOLS1993309}, Lasserre~\cite{Lasserre2011,Lasserre2012}, Sommariva–Vianello~\cite{SommarivaVianello2020}, and Gustafsson~\cite{Gustafsson2023}.

{
A closely related line of work is the theory of spherical designs, where one seeks discrete measures on the sphere that reproduce all moments up to a given degree. In contrast to general quadrature rules, spherical designs impose equal weights and strong symmetry constraints. The breakthrough result of Bondarenko, Radchenko, and Viazovska~\cite{bondarenko2013optimal} shows that such designs exist with $N = \mathcal{O}(t^d)$ points, matching the natural lower bound. From the perspective of moment cones, this can be viewed as a highly structured counterpart to Carathéodory-type bounds, where the representing measures are required to be uniform. While our methods do not directly apply in this setting, it would be interesting to explore whether the optimization viewpoint developed here can shed further light on the minimality and structure of spherical designs.}

\paragraph{Structure of the Paper:} Section~\ref{sec:optimization} introduces the optimization framework, Section~\ref{sec:plane} proves Theorem \ref{prop:planecurve} giving the bounds for plane algebraic curves, and Section~\ref{sec:rational} establishes Theorems \ref{thm:quadraturerational} and \ref{thm:quadraturerationaleven} for rational curves in arbitrary dimensions.

\section{Optimization approach}\label{sec:optimization}

We now describe the optimization framework underlying our results. Fix $s \in \mathbb{N}$ and let $\mu$ be a {positive} Borel measure on $\R^n$ with moments  
\[
m_{\alpha} := \int x^\alpha \, d\mu
\]
where $\alpha=(\alpha_1,\dots,\alpha_n)$ and $x^\alpha=x_1^{\alpha_1}\cdots x_n^{\alpha_n}$, finite for all $|\alpha|=\alpha_1+\dots+\alpha_n \le s$. A quadrature rule of strength $s$ with at most $N$ nodes corresponds to a point  
\[
q = (\omega, x) \in \R^N_{\ge 0} \times (\operatorname{supp}(\mu))^N
\]
satisfying the moment-matching conditions
\[
\sum_{i=1}^N \omega_i x_{1,i}^{\alpha_1}\dots x_{n,i}^{\alpha_n}= m_{\alpha} \qquad \text{for all } |\alpha|\le s.
\]
The set of all such points will be denoted by  
\[
Q(\mu)_s = \left\{ (\omega, x) \in \R^N_{\ge 0} \times (\operatorname{supp}(\mu))^N : 
\sum_{i=1}^N \omega_i x_{1,i}^{\alpha_1}\dots x_{n,i}^{\alpha_n}= m_{\alpha} , \, \forall\, |\alpha| \le s \right\}.
\]
Carathéodory’s theorem guarantees that for sufficiently large $N$, this set is nonempty; moreover, the \emph{minimal} number of nodes corresponds to the minimal number of positive weights in any such representation.

A key idea, adapted from  \cite{RS}, is to study the problem
\begin{equation}\label{eq:opt}
  \min \, P(\omega,x) \qquad \text{subject to } (\omega,x)\in Q(\mu)_s
\end{equation}
for a carefully chosen penalty function $P$. For instance, taking $P(\omega,x)=\sum_i \omega_i$ or $\sum_i (x_{1,i}^2+\dots+x_{n,i}^2)$ forces nodes to cluster or spread out in a controlled way. Optimal solutions of \eqref{eq:opt} automatically produce quadrature rules with as few nodes as possible, because redundant nodes either receive zero weight or coincide with others.  

This optimization viewpoint has two main advantages. First, it translates the \emph{existence} problem for quadrature rules into a concrete minimization problem, to which tools from real algebraic geometry and polynomial optimization apply. Second, the geometry of the support of $\mu$, in our case, plane or general rational curves, enters the analysis explicitly via the parametrization of the feasible set $Q(\mu)_s$.  

In the \emph{even-degree} case $2s$, the moment cone approach alone provides bounds on the number of nodes through Carathéodory numbers. However, in the \emph{odd-degree} case $2s-1$, relevant for Gaussian quadrature, these bounds are no longer optimal, and solving \eqref{eq:opt} becomes essential. As we show in the next sections, different geometric assumptions on the support, such as compactness, number of places at infinity, or rational parametrizations, lead to explicit node bounds improving upon classical results.

\section{Quadrature on plane curves}\label{sec:plane}

{We first focus on the case of plane curves. Let $\mathcal{C}\subset \R^2$ be a real algebraic curve of degree $d$ defined by $F(x,y)=0$. Carath\'eodory's theorem applied to the quotient space $\R[x,y]/(F)$ shows that any quadrature rule of strength $s=d+r$ supported on $\mathcal{C}$ requires at most}
\[
\binom{d+1}{2} + (r+1)d \;=\; \dim \bigl(\R[x,y]/(F)\bigr)_{s}
\]
{nodes.}  

Using optimization methods, Schweighofer and Riener~\cite{RS} improved this direct bound for plane algebraic curves. They showed that for any algebraic curve of degree $d$, there exists a quadrature rule supported entirely on $\mathcal{C}$ with at most $sd$ nodes and positive weights, which is exact for all polynomials of degree at most $2s-1$. More recently, Zalar~\cite{zalar} refined this result for certain families of curves, providing strictly smaller node bounds in these special cases.

\begin{definition}
 {We say that $\mathcal C\subset \R^n$ has a place at infinity, if its projectivization $   \mathcal C(\mathbb P)\subset \mathbb P^n$ intersects the real hyperplane at infinity $H=\{(t_1:\dots:t_n:0)\in\mathbb P^n\mid t_i\in \R\} $.  }
\end{definition}

With the  following Theorem we improve the  bound for plane curves.

\begin{thm}\label{prop:planecurve}
Let $\mathcal{C}\subset \R^2$ be a smooth {and irreducible} real algebraic curve of degree $d$ with at most $t \leq d$ {distinct} places at infinity {counted without multiplicities, i.e., $\#(\mathcal C(\mathbb P)\cap H)\leq t$}. Then every measure supported on $\mathcal C$ admits a quadrature rule of strength $2s-1$ with at most
\[
N \;\le\; ds - \left\lceil \tfrac d2 \right\rceil + 1 + 2t
\]
nodes.
\end{thm}
\begin{proof}
Let $F \in \R[x,y]$ be the defining polynomial of $\mathcal{C}$. We consider the optimization problem
\[
\begin{cases}
\min \; \sum_{i=1}^N \omega_i, \\[2pt]
\sum_{i=1}^N \omega_i x_i^a y_i^b = m_{ab}, & 1 \le a+b \le 2s-1, \\
F(x_i,y_i)=0, & i=1,\dots,N, \\
x_i^2 + y_i^2 \le R,
\end{cases}
\]
where $R>0$ is chosen so that the feasible set is non-empty and compact.   
Existence of such an $R$  follows from the existence of quadrature rules per se \cite{richter1957parameterfreie}. 

Intersections with the circle $x^2+y^2=R$ contribute at most $2t$ additional points. We first analyse the optimal pairs $(x_i,y_i)$ in the interior of the disk.

By the method of Lagrange multipliers, there exist $\lambda_{ab}$ and $\lambda_i$ such that the Lagrangian  
\[
\mathcal{L} = \sum_{i=1}^N \omega_i 
+ \sum_{0\le a+b\le 2s-1} \lambda_{ab} \Bigl(\sum_{i=1}^N \omega_i x_i^a y_i^b - m_{ab}\Bigr) 
+ \sum_{i=1}^N \lambda_i F(x_i,y_i)
\]
is stationary with respect to all variables, i.e.,
{\begin{align*}
    0=&\frac{\partial \mathcal L}{\partial \lambda_{ab}}=\sum_{i=1}^N \omega_i x_i^a y_i^b - m_{ab}, \\    
    0=&\frac{\partial \mathcal L}{\partial \lambda_{i}}=F(x_i,y_i),\\
    0=&\frac{\partial \mathcal L}{\partial \omega_i}=1+\sum_{0\le a+b\le 2s-1} \lambda_{ab} x_i^a y_i^b,  \\
    0=&\frac{\partial \mathcal L}{\partial x_i}=\sum_{0\le a+b\le 2s-1} \lambda_{ab} \omega_i a x_i^{a-1} y_i^b+\frac{\partial F}{\partial x_i} (x_i,y_i),\\
    0=&\frac{\partial \mathcal L}{\partial y_i}=\sum_{0\le a+b\le 2s-1} \lambda_{ab} \omega_i b x_i^{a} y_i^{b-1}+\frac{\partial F}{\partial y_i} (x_i,y_i).
\end{align*}}
Therefore, the optimal nodes $(x_i,y_i)$ satisfy  
\[
G(x,y) = 1 + \sum_{1\le a+b\le 2s-1} \lambda_{ab} x^a y^b = 0
\]
and $F(x,y)=0$. Moreover, the gradient condition  
\[
\nabla G(x_i,y_i) = - \frac{\lambda_i}{\omega_i} \nabla F(x_i,y_i)
\]
implies that the two curves $F=0$ and $G=0$ have parallel tangent lines at each node, hence each intersection point has multiplicity at least two.  

By Bézout’s theorem, $F$ and $G$ have at most $d(2s-1)$ intersection points in $\C^2$, so at most  
\[
\Bigl\lfloor \tfrac{d(2s-1)}{2} \Bigr\rfloor = ds - \left\lceil \tfrac d2 \right\rceil
\]
real nodes of multiplicity at least two. 

Finally, since the equation for the zeroth moment $\sum_i \omega_i = m_{00}$ is not imposed, we may add one more node at the origin with the remaining weight to correct the total mass. Adding the potential solutions at the boundary of the disk yields the bound  
\[
N \;\le\; ds - \left\lceil \tfrac d2 \right\rceil + 1 + 2t,
\]
as claimed.
\end{proof}

\begin{remark}
    {The irreducibility assumption in Theorem \ref{prop:planecurve} plays mostly a technical role for the proof. The result extends naturally also to non-irreducible curves by applying the same reasoning to each irreducible component separately. The resulting bound differs only in lower-order terms due to the ceiling.}
\end{remark}

\begin{cor}
If $\mathcal C \subset \R^2$ is compact, then every measure supported on $\mathcal C$ admits a quadrature rule of strength $2s-1$ with at most  
\[
N \;\le\; ds - \left\lceil \tfrac d2 \right\rceil + 1
\]
nodes.
\end{cor}

\begin{proof}
When $\mathcal C$ is compact, there are no places at infinity, so the term $2t$ in Theorem~\ref{prop:planecurve} vanishes. The remaining argument is identical, yielding the stated bound.
\end{proof}

\begin{remark}
For a plane algebraic curve of degree $d$, Riener and Schweighofer~\cite{RS} proved that there exists a quadrature rule of strength $2s-1$ with at most $ds$ nodes supported on the curve. 
Our bound in Theorem~\ref{prop:planecurve} improves this to  
\[
N \;\le\; d\,s - \left\lceil \tfrac d2 \right\rceil + 1 + 2t.
\]
The difference between the two bounds is at least  
\[
d\,s - \Bigl(d\,s - \left\lceil \tfrac d2 \right\rceil + 1 + 2t\Bigr) 
= \left\lceil \tfrac d2 \right\rceil - 1 - 2t,
\]
so for any fixed curve of degree $d$, the improvement is independent of the strength.
\end{remark}

\section{Quadrature on rational curves}\label{sec:rational}

In this section, we focus on \emph{rational curves}, i.e., genus zero real algebraic curves admitting a parametrization by rational functions. Such curves can be highly singular and may be embedded in ambient spaces of arbitrary dimension. Without imposing smoothness or restricting to the plane, we derive explicit bounds for quadrature rules supported on rational curves that extend the result of \Cref{prop:planecurve} beyond the smooth planar case. Our bounds recover, for rational plane curves, results comparable to those obtained by Zalar~\cite{zalar}, while applying equally to arbitrary rational curves in any dimension, thus covering a far broader class of singular and higher-dimensional situations.

Let $\mathcal C \subset \R^n$ be a rational curve of degree $d$, i.e., a real algebraic curve admitting a parametrization
\begin{equation}\label{eq:parametrization}
\varphi : \R \dashrightarrow \R^n, 
\qquad 
t \mapsto \left(\frac{\varphi_1(t)}{\varphi_0(t)}, \dots, \frac{\varphi_n(t)}{\varphi_0(t)} \right),
\end{equation}
with $\varphi_i \in \R[t]_{\leq d}$.  

Since $\mathcal C$ is no longer contained in a plane, we take as a natural starting point the Carathéodory bound for quadrature rules of strength $s$ in $\R^n$,
\[
N \;>\; \binom{n+s}{s} = \dim \R[x_1,\dots,x_n]_s,
\]
which gives the maximal number of nodes one might expect before exploiting the specific geometry of $\mathcal C$. Our goal is to improve substantially on this bound by using the rational parametrization \eqref{eq:parametrization}.
\begin{lemma}
Let $\varphi(t) = (t, t^{d-1}, t^d)$ be a parametrization of a curve in $\R^3$.  
Then the map
\[
\psi:\R[x,y,z]_{\le s} \longrightarrow \R[t]_{\le ds},
\qquad
p(x,y,z) \longmapsto p(\varphi(t))
\]
is surjective whenever $s \ge d$.
\end{lemma}

\begin{proof}
Every monomial $x^a y^b z^c$ with $a+b+c \le s$ maps to  
\[
x^a y^b z^c \longmapsto t^{a + b(d-1)+ c d}.
\]
For fixed $b,c$ with $b+c \le s$, the exponents appearing are
\[
\{a + b(d-1)+ c d : 0 \le a \le s-b-c \}
= \{k \in \N : k_{\min} \le k \le k_{\max}\},
\]
where $k_{\min} = b(d-1)+cd$ (take $a=0$) and $k_{\max}=b(d-1)+cd + (s-b-c)$ (take $a=s-b-c$).  
Set $\sigma=b+c$. Then $k_{\min} = \sigma d - b$ and $k_{\max} = \sigma(d-1)-b+s$.  

For fixed $\sigma$, taking all $0 \le b \le \sigma$ produces the union of intervals
\[
I_\sigma = \bigcup_{b=0}^\sigma [\sigma d-b,\ \sigma(d-1)-b+s]\cap\N
= [\,\sigma d - \sigma,\ \sigma(d-1)+s\,]\cap \N,
\]
since the lower bound comes from $b=\sigma$ and the upper bound from $b=0$.  
This uses the assumption $s \ge \sigma$ to ensure the intervals overlap without gaps.

Now vary $\sigma$ from $0$ to $s$.  
The intervals satisfy  
\[
I_\sigma \cup I_{\sigma+1} 
= [\,\sigma d - \sigma,\ (\sigma+1)(d-1)+s\,]\cap \N,
\]
and the assumption $s \ge d$ ensures that the overlaps fill all integers up to $ds$.  
Hence
\[
\bigcup_{\sigma=0}^s I_\sigma = [0, ds]\cap \N,
\]
so every monomial $t^k$ with $0 \le k \le ds$ appears in the image.  
Thus $\psi$ is surjective.
\end{proof}

\begin{cor}\label{cor:surjective}
    Let $\varphi:\R\to \R^n$, {$n\ge 3$}$, \varphi(t)=(\varphi_1(t),\dots,\varphi_n(t))$ and $\deg \varphi_i\leq d$. If $\varphi_i$ are generic and $s\geq d$, then the map $$\psi:\R[x_1,\dots,x_n]_{\leq s}\to \R[t]_{ds},\ p(x_1,\dots,x_n)\mapsto p(\varphi(t))
    $$
    is surjective.
\end{cor}
\begin{proof}
    First notice that for $n=3$ we have $\varphi_j(t)=\sum_{i=0}^d\alpha_{ij}t^i$. We showed that the matrix of the map $\psi$ has full rank for a specific choice of $\alpha_{ij}$, and thus it has full rank for a generic choice of $\alpha_{ij}$.

    To extend for $n\geq 3$, it is enough to see that $\varphi=(t,t^{d-1},t^d,1,\dots,1)$ already generates the target space.
\end{proof}

\begin{remark}\label{rmk:n=2 notsurjective}
    For $n=2$ the map is rank defective. The reason is that the curve defined by $\varphi$ in $\R^2$ has an equation given by $f$ of degree $d$, and $I=(f)$ is the kernel of the map. In particular, when $s=d$, the kernel is one-dimensional.
\end{remark}
\begin{definition}
Let $\mu$ be a finite nonnegative Borel measure supported on a curve $\mathcal C\subset\R^n$.
We say that $\mu$ is \emph{degenerate at degree $k$} if the quadratic form
\[
q_k(f)\ :=\ \int_{\mathcal C} f(x)^2\,d\mu(x)\qquad(f\in\R[x_1,\dots,x_n]_{\le k})
\]
is not positive definite on $\R[x_1,\dots,x_n]_{\le k}$, i.e., if there exists a nonzero $f\in\R[x_1,\dots,x_n]_{\le k}$ with $\int f^2\,d\mu=0$.  
Equivalently, the moment matrix $M_k(\mu)$ has a nontrivial kernel. 
\end{definition}

\begin{lemma}
Let $n\ge 3$, let $\mathcal C\subset\R^n$ be a rational curve with a polynomial parametrization $\varphi:\R\to\R^n$, $\varphi(t)=(\varphi_1(t),\dots,\varphi_n(t))$, and let $D:=\max_i\deg\varphi_i$. Let $\mu$ be a finite nonnegative Borel measure supported on $\mathcal C$, and let $\nu:=\varphi^{-1}_*\mu$ be the pullback measure on $\R$. Assume that the composition map
\[
\Psi_k:\ \R[x_1,\dots,x_n]_{\le k}\longrightarrow \R[t]_{\le Dk},\qquad p\longmapsto p\circ\varphi
\]
is surjective (e.g., this holds for generic $\varphi$ when $n\ge 3$). Then $\mu$ is degenerate at degree $k$ if and only if $\nu$ is degenerate at degree $Dk$.
\end{lemma}

\begin{proof}
 If $\mu$ is $k$-degenerate, there exists $0\neq p\in\R[x_1,\dots,x_n]_{\le k}$ with $\int_{\mathcal C}p^2\,d\mu=0$. Set $g:=p\circ\varphi\in\R[t]_{\le Dk}$. By the change of variables along $\varphi$,
\[
\int_{\R} g(t)^2\,d\nu(t)\ =\ \int_{\R} p(\varphi(t))^2\,d\nu(t)\ =\ \int_{\mathcal C} p(x)^2\,d\mu(x)\ =\ 0,
\]
so $\nu$ is $Dk$-degenerate.

 If $\nu$ is $Dk$-degenerate, there exists $0\neq g\in\R[t]_{\le Dk}$ with $\int g^2\,d\nu=0$. By surjectivity of $\Psi_k$, pick $p\in\R[x_1,\dots,x_n]_{\le k}$ with $g=p\circ\varphi$. Then
\[
\int_{\mathcal C} p(x)^2\,d\mu(x)\ =\ \int_{\R} p(\varphi(t))^2\,d\nu(t)\ =\ \int_{\R} g(t)^2\,d\nu(t)\ =\ 0,
\]
so $\mu$ is $k$-degenerate.
\end{proof}
\begin{remark}
If $\mu$ is degenerate at degree $k$, then there exists a nonzero $p\in\R[x_1,\dots,x_n]_{\le k}$ with $\int_{\mathcal C} p^2\,d\mu=0$, hence $p=0$ $\mu$-almost-everywhere and $\supp(\mu)\subseteq Z(p)\cap \mathcal C$. Equivalently, the moment matrix $M_k(\mu)$ has nontrivial kernel and $\rank M_k(\mu)<\dim\R[x_1,\dots,x_n]_{\le k}$. Under the parametrization $\varphi$, the corresponding $g=p\circ\varphi\in\R[t]_{\le Dk}$ vanishes $\nu$-almost-everywhere, so $\supp(\nu)\subseteq Z(g)$; in particular, the support of $\mu$ is contained in the image of a finite (algebraic) subset of the parameter line.
\end{remark}

\begin{thm}\label{thm:curve-to-line}
Let $\mathcal C\subset \R^n$ be a rational curve of degree $d$ with a polynomial parametrization
$\varphi:\R\to \R^n,\ t\mapsto (\varphi_1(t),\dots,\varphi_n(t))$, with $\deg \varphi_i\le d$.
Assume $n\ge 3$, the coefficients of $\varphi$ are generic, $s\ge d$, $\operatorname{supp}(\mu)\subseteq \varphi(\R)$ and $\mu$ not degenerate in degree $s$.
Then the minimal number of nodes in a quadrature rule of strength $s$ for $Q(\mu)_s$ restricted to $\mathcal C$
is bounded below by the minimal number of nodes in a quadrature rule of strength $sd$ on the line.
\end{thm}

\begin{proof}
Suppose $Q(\mu)_s$ admits a quadrature rule on $\mathcal C$ with nodes $x^{(1)},\dots,x^{(k)}\in \mathcal C$ and
positive weights $\omega_1,\dots,\omega_k$, i.e.
\[
\int_{\mathcal C} p(x)\,d\mu(x)\;=\;\sum_{i=1}^k \omega_i\,p\bigl(x^{(i)}\bigr)
\quad\text{for all }p\in \R[x_1,\dots,x_n]_{\le s}.
\]
Since $\operatorname{supp}(\mu)\subseteq \varphi(\R)$, pick $t_i\in\R$ with $\varphi(t_i)=x^{(i)}$.
Define the pullback measure $\mu' := \varphi^{-1}_*\mu$ on $\R$ by $\mu'(B)=\mu(\varphi(B))$.
Then, for any $p\in \R[x_1,\dots,x_n]_{\le s}$,
\[
\int_{\R} p(\varphi(t))\, d\mu'(t)\;=\;\int_{\mathcal C} p(x)\,d\mu(x)
\;=\;\sum_{i=1}^k \omega_i\,p\bigl(\varphi(t_i)\bigr).
\]
Set $g(t):=p(\varphi(t))$. Since $\deg g \le s\cdot d$, we have $g\in \R[t]_{\le sd}$.

We claim that for $n\ge 3$, generic $\varphi$, and $s\ge d$, every polynomial
$g\in \R[t]_{\le sd}$ arises as $g(t)=p(\varphi(t))$ for some $p\in \R[x_1,\dots,x_n]_{\le s}$.
This is the surjectivity of the restriction map
\[
\psi:\R[x_1,\dots,x_n]_{\le s}\longrightarrow \R[t]_{\le sd},\qquad p\longmapsto p(\varphi(t)),
\]
which holds generically in the coefficients of $\varphi$ when $n\ge 3$ (and fails for $n=2$ because of
the curve relation) by Corollary \ref{cor:surjective}.

Hence the equality above holds for all $g\in \R[t]_{\le sd}$. Therefore, the same $k$ nodes $t_1,\dots,t_k$
with weights $\omega_1,\dots,\omega_k$ form a quadrature rule of strength $sd$ for the univariate measure $\mu'$:
\[
\int_{\R} g(t)\, d\mu'(t)\;=\;\sum_{i=1}^k \omega_i\, g(t_i)\qquad\forall\, g\in \R[t]_{\le sd}.
\]
Since $\mu'$ is nondegenerate in degree $sd$ and Gaussian quadrature is optimal for nondegenerate measures (see \cite[Thm 9.2]{RS}) \, a quadrature rule exact for degrees up to $m$ requires
at least $\lceil (m+1)/2\rceil$ nodes. With $m=sd$, this gives $k\ge \lceil (sd+1)/2\rceil$.
Equivalently, if $sd=2\ell-1$ then $k\ge \ell$, and if $sd=2\ell$ then $k\ge \ell+1$.
\end{proof}

\begin{cor}\label{cor:minimalnumberof nodes}
Let $\mathcal C\subset \R^n$ be a degree $d$ rational curve with a generic polynomial parametrization and $n\geq3$. Suppose $2\ell, 2s-1\geq d$, then
any quadrature rule in $Q_{2s-1}(\mu)$ on $\mathcal C\subset \R^n$ has at least $ds-\lceil\frac{d}{2} \rceil+1$ nodes, and
any quadrature rule in $Q_{2\ell}(\mu)$ on $\mathcal C$ has at least $d\ell+1$ nodes.
\end{cor}

\begin{thm}\label{thm:quadraturerational}
    Let $\mathcal C\subset \mathbb R^n$ be a degree $d$ affine rational curve with parametrization $$\varphi:t\mapsto \left(\frac{\varphi_1(t)}{\varphi_0(t)},\dots,\frac{\varphi_n(t)}{\varphi_0(t)}\right),$$ where $\deg \varphi_i=d_i\leq d$ and $D=\max\{d_1,\dots,d_n\}$. Suppose $\varphi_0$ has $p\geq 0$ distinct real zeros. Then every measure supported on $\mathcal C$ admits a quadrature rule of strength $2s-1$ with at most $$N\leq Ds-\left\lceil \frac{D}{2}\right\rceil+p+1,$$
    nodes.
\end{thm}
\begin{proof}
    Let $\xx_i=(x_{i_1},\dots,x_{i_n})\in\mathcal C\subset\R^n$ and $P_\alpha(\omega_i,\xx_i)=\sum_{i=1}^N\omega_ix_{i_1}^{\alpha_1}\cdots x_{i_n}^{\alpha_n}$. We recall that since $\deg(\mathcal C)=d$, then for at least one $i\in\{0,\dots,n\}$ we have $d_i=d$. {Using the parametrization of $\mathcal C$, we write $x_i = \varphi(t_i)$, and define $P_\alpha(\omega_i, t_i) := P_\alpha(\omega_i, \varphi(t_i)) =\sum_{i=1}^N\omega_i\left(\frac{\varphi_1}{\varphi_0}(t_i)\right)^{\alpha_1}\cdots\left(\frac{\varphi_n}{\varphi_0}(t_i)\right)^{\alpha_n}$. Thus, a quadrature rule with measure $\mu$ in $\R^n$ restricted to $\mathcal C$ corresponds to $(\omega_i,s_i)\in \R^N_{\geq 0}\times \R^N$ such that $P_{\alpha}(\omega_i,s_i)=m_\alpha$ for all $|\alpha|\leq 2s-1$.} 
    
    Let $Z=\{z_1,\dots,z_p\}$ the real zeros of $\varphi_0$ and $h(t)=t^2+\sum_{i=1}^p\frac{1}{(t-z_i)^2}$. We consider the following optimisation problem $$
    \begin{cases}
        \min \sum^N_{i=1} h(t_i),\\
        P_\alpha(\omega_i,t_i)-m_\alpha=0,\ |\alpha|\leq 2s-1.
    \end{cases}
    $$
    Using a singular version of Lagrange multipliers \cite[Theorem
6.1.1, Remark 6.1.2(iv)]{Cla90} on $U=\mathbb R\setminus Z$, there exists $\lambda_\alpha$, $|\alpha|\leq 2s-1$ such that $$\mathcal L(\lambda_{\alpha}, t_i,\omega_i)={\lambda_0}\sum_{i=1}^N h(t_i)+\sum_{|\alpha|\leq 2s-1} \lambda_{\alpha}(P_{\alpha}(\omega_i,t_i)-m_{\alpha}),$$ and the local minima are the solutions of 
        \begin{align*}
             0=&\frac{\partial\mathcal L}{\partial \lambda_{\alpha}}=P_{\alpha}(\omega_i,t_i)-m_{\alpha},\\
     0=&\frac{\partial\mathcal L}{\partial \omega_i}=\sum_{|\alpha|\leq 2s-1}\lambda_{\alpha}\varphi_1(t_i)^{\alpha_1}\cdots\varphi_n(t_i)^{\alpha_n},\\
     0=&\frac{\partial\mathcal L}{\partial t_i}={\lambda_0}h'(t_i) +\sum_{|\alpha|\leq 2s-1}\omega_i\lambda_{\alpha}[\alpha_1\varphi_1(t_i)^{\alpha_1-1}\varphi_1'(t_i)\varphi_2(t_i)^{\alpha_2}\cdots \varphi_n(t_i)^{\alpha_n}+\cdots\\   &\cdots +\alpha_n\varphi_1(t_i)^{\alpha_1}\cdots \varphi_{n-1}(t_i)^{^{\alpha_{n-1}}}\varphi_n(t_i)^{\alpha_n-1}\varphi_n'(t_i)].
    \end{align*}
    Let $H(t)=\sum \lambda_{\alpha}\varphi_1(t)^{\alpha_1}\cdots\varphi_n(t)^{\alpha_n}$, then the two last equations amount to $H(t_i)=0$ and $H'(t_i)=-\frac{{\lambda_0}}{\omega_i}h'(t_i)$. 
    
    {We first deal with the case ${\lambda_0= 0}$. Notice this amounts to $H(t_i)=0$ and $H'(t_i)=0$ for all $i=1,\dots, N$}. Thus, each node is a multiplicity-two root of $H$, so there are at most $\left\lfloor \frac{\deg H}{2}\right\rfloor= \left\lfloor \frac{(2s-1)D}{2}\right\rfloor= Ds-\left\lceil\frac{D}{2}\right\rceil$ nodes.

   {If $\lambda_0\neq0$, we may assume without lost of generality, $\lambda_0=1$}. Therefore, we reduce the problem to analysing the signs of the derivatives of $H$ and $h$ at the roots $t_i$ of $H$. Notice $h''\geq 0$, so $h'$ is monotone, so $h\geq0$ is concave is each interval $(z_i,z_{i+1})$, i.e. there exists a unique minima $y_i\in(z_i,z_{i+1})$ (with $z_0=-\infty$,$z_{p+1}=\infty$), moreover both $H$ and $h$ are $C^\infty$ on $U$. Let $T_i=\{t_j\in\R|H(t_j)=0, t_j\in (z_i,z_{i+1})\}$ and $T_{i_0}=\{t_j\in\R|H(t_j)=0, t_j\in (z_i,y_i)\}$, $T_{i_1}=\{t_j\in\R|H(t_j)=0, t_j\in (y_i,z_{i+1})\}$. Notice $\sum_{i=0}^p \#T_i\leq (2s-1)D=\deg(H)$.

    It follows that $t_j\in T_{i_0}$ is a node if and only if $H'(t_j)\geq0$, likewise $t_j\in T_{i_1}$ is a node only if $H'(t_j)\leq 0$. We first show that the maximum number of nodes in $T_i$ depends only on the parity of $\#T_i$.

    Indeed, if $\#T_i$ is odd, then without lost of generality $\#T_{i_0}$ is even and $\#T_{i_1}$ is odd. Suppose $t_{j_1}$ is the smallest and $t_{j_k}$ the largest roots in $T_{i_0}$. Then there are at most $\#T_{i_0}/2$ roots with nonnegative derivative, as consecutive roots either change the sign of the derivative, or are not a simple root. On the other hand, by the same argument $T_{i_1}$ has at most $\lceil\#T_{i_1}/2 \rceil$ roots with nonpositive derivative. Thus the maximum number of nodes in $T_i$ is $\lceil\#T_i/2 \rceil$

    Finally, if $\#T_i$ is even, then $\#T_{i_0},\#T_{i_1}$ have the same parity. If both are even, then the number of nodes is upper bounded by $\#T_i/2$. On the other hand, if they are both odd, then the maximum number of nodes in each $T_{i_j}$ is $\lceil \#T_{i_j}/2\rceil$, thus the maximal number of nodes in $T_i$ is $\lceil \#T_{i_0}/2\rceil+\lceil \#T_{i_1}/2\rceil= \#T_{i}/2+1$.

     Suppose $T_0,\dots,T_j$ have even cardinality, while $T_{j+1},\dots,T_{p}$ have odd cardinality. Thus the number of nodes is upper bounded by $$
    \sum_{i=0}^j \left(\frac{\#T_i}{2}+1\right)+\sum_{i=j+1}^{p}\left(\frac{\#T_i}{2}+\frac{1}{2}\right)\leq\frac{(2s-1)D}{2} +\frac{p+j}{2}+1=sD+\frac{p+j-D}{2}+1,
    $$
    which is maximum when $j=p$ if $D$ is even, and $j=p-1$ when $D$ is odd. In both cases, the previous equality amounts to $
    sD-\left\lceil \frac{D}{2}\right\rceil +p+1.
    $

\end{proof}
\begin{remark}\label{rmk:optimaln>3}
    Notice that when $\varphi_0$ has no real zeroes, we have $j=0=p$ if $D$ is even, and $j=-1=p-1$ if $D$ is odd. \Cref{thm:quadraturerational} provides the respective bounds $Ds-\frac{D}{2}+1$ and $Ds-\lfloor D/2\rfloor$. We notice that these are the optimal bounds on the number of nodes for $Q_{(2s-1)D}(\mu)$ on the line. In particular, when  $n\geq 3$, $\varphi_0$ is a constant, and the others $\varphi_j$ are generic, then Corollary \ref{cor:minimalnumberof nodes} shows this is the optimal number of nodes.
\end{remark}

\begin{thm}\label{thm:quadraturerationaleven}
Let $\mathcal C \subset \mathbb R^n$ be a degree $d$ affine rational curve with parametrization
\[
t \mapsto \left(\frac{\varphi_1(t)}{\varphi_0(t)}, \dots, \frac{\varphi_n(t)}{\varphi_0(t)}\right),
\]
where $\deg \varphi_i = d_i \le d$ and $D = \max\{d_1,\dots,d_n\}$.
Suppose $\varphi_0$ has $p \ge 0$ distinct real zeros.  
Then every measure supported on $\mathcal C$ admits a quadrature rule of strength $2s$ with at most
\[
N \;\le\; D s + p + 1
\]
nodes.
\end{thm}

\begin{proof}
{Let $\xx_i=(x_{i_1},\dots,x_{i_n})\in\mathcal C\subset\R^n$ and $P_\alpha(\omega_i,\xx_i)=\sum_{i=1}^N\omega_ix_{i_1}^{\alpha_1}\cdots x_{i_n}^{\alpha_n}$. We recall that since $\deg(\mathcal C)=d$, then for at least one $i\in\{0,\dots,n\}$ we have $d_i=d$. {Using the parametrization of $\mathcal C$, we write $x_i = \varphi(t_i)$, and define $P_\alpha(\omega_i, t_i) := P_\alpha(\omega_i, \varphi(t_i)) =\sum_{i=1}^N\omega_i\left(\frac{\varphi_1}{\varphi_0}(t_i)\right)^{\alpha_1}\cdots\left(\frac{\varphi_n}{\varphi_0}(t_i)\right)^{\alpha_n}$. Thus, a quadrature rule with measure $\mu$ in $\R^n$ restricted to $\mathcal C$ corresponds to $(\omega_i,s_i)\in \R^N_{\geq 0}\times \R^N$ such that $P_{\alpha}(\omega_i,s_i)=m_\alpha$ for all $|\alpha|\leq 2s$.}}

Let $Z = \{z_1,\dots,z_p\}$ be the real zeros of $\varphi_0$, and define  
\[
h(t) = t^2 + \sum_{i=1}^p \frac{1}{(t-z_i)^2}.
\]
We consider the optimization problem
\[
\begin{cases}
\min \sum_{i=1}^N h(t_i), \\
P_\alpha(\omega_i,t_i) - m_\alpha = 0 \quad \text{for all } |\alpha|\le 2s,
\end{cases}
\]
where
\[
P_\alpha(\omega_i,t_i)=\sum_{i=1}^N \omega_i \left(\frac{\varphi_1(t_i)}{\varphi_0(t_i)}\right)^{\alpha_1}
\dotsm \left(\frac{\varphi_n(t_i)}{\varphi_0(t_i)}\right)^{\alpha_n}.
\]

By the singular Lagrange multiplier rule \cite[Theorem 6.1.1, Remark 6.1.2(iv)]{Cla90} applied on $U=\mathbb R\setminus Z$,
there exist multipliers $\lambda_\alpha$ with $|\alpha|\le 2s$ such that the Lagrangian
\[
\mathcal L(t_i,\omega_i,\lambda_\alpha) =
{\lambda_0}\sum_{i=1}^N h(t_i) +
\sum_{|\alpha|\le 2s} \lambda_\alpha \bigl(P_\alpha(\omega_i,t_i)-m_\alpha\bigr)
\]
is stationary at any local minimizer. Differentiating gives the conditions
\begin{align*}
\frac{\partial \mathcal L}{\partial \lambda_\alpha} &= P_\alpha(\omega_i,t_i)-m_\alpha = 0, \\
\frac{\partial \mathcal L}{\partial \omega_i} &= \sum_{|\alpha|\le 2s} \lambda_\alpha \,\varphi_1(t_i)^{\alpha_1}\dotsm \varphi_n(t_i)^{\alpha_n} = 0, \\
\frac{\partial \mathcal L}{\partial t_i} &= {\lambda_0} h'(t_i) + 
\omega_i \sum_{|\alpha|\le 2s} \lambda_\alpha \sum_{k=1}^n 
\alpha_k \varphi_1(t_i)^{\alpha_1}\dotsm \varphi_k(t_i)^{\alpha_k-1}\varphi_k'(t_i)
\dotsm \varphi_n(t_i)^{\alpha_n} = 0.
\end{align*}

Let
\[
H(t) = \sum_{|\alpha|\le 2s} \lambda_\alpha\,\varphi_1(t)^{\alpha_1}\dotsm \varphi_n(t)^{\alpha_n}.
\]
Then the second and third equations become
\[
H(t_i)=0 \qquad\text{and}\qquad H'(t_i)=-\frac{{\lambda_0}}{\omega_i}h'(t_i).
\]

{As in the proof of \Cref{thm:quadraturerational}, we have two cases: $\lambda_0=0$ and $\lambda_0\neq 0$. In the first case, this implies that $H(t_i)=H'(t_i)=0$ for all $i=1,\dots,N$}. So each $t_i$ is a multiplicity-two root, thus there are at most $\left\lfloor\frac{\deg H}{2} \right\rfloor=sD$ nodes. Otherwhise, we may assume $\lambda_0=1$.

In order to estimate the number of nodes we observe that 
on each interval $(z_i,z_{i+1})$, the function $h$ is smooth and strictly convex, with a unique minimizer $y_i\in(z_i,z_{i+1})$.
Let
\[
T_i = \{t_j \in (z_i,z_{i+1}) \mid H(t_j)=0\}
\]
be the set of roots of $H$ in that interval. Since $\deg H \le 2sD$, we have
\[
\sum_{i=0}^p \#T_i \;\le\; 2sD.
\]

A node arises whenever $H'(t_j)$ has the correct sign relative to $h'(t_j)$:
points in $(z_i,y_i)$ are counted if $H'(t_j)\ge 0$, those in $(y_i,z_{i+1})$ if $H'(t_j)\le 0$.
As in the proof of \Cref{thm:quadraturerational}, the maximal number of nodes from $T_i$ is
\[
\begin{cases}
\#T_i/2 + 1 & \text{if $\#T_i$ is even},\\[2pt]
\lceil \#T_i/2\rceil & \text{if $\#T_i$ is odd}.
\end{cases}
\]

This is maximized when all $\#T_i$ are even, giving
\[
N \;\le\; \sum_{i=0}^p \left(\frac{\#T_i}{2}+1\right)
\;\le\; \frac{2sD}{2} + p + 1
\;=\; sD + p + 1.
\]\qedhere
\end{proof}

\begin{remark}
    Notice, as in Remark \ref{rmk:optimaln>3}, if $n\geq 3$, $\varphi_0$ is constant and $\varphi_i$ are generic for $i\neq 0$, then the bound obtained is optimal.
\end{remark}

\begin{prop}\label{prop:ttd}
     Let $\mathcal C\subset \R^2$ be the curve defined by $y=x^d$. {Suppose $Q(\mu)_s$ has a quadrature rule on $\mathcal C$ such that all nodes $(x_i,x_i^d)$ are such that $x_i\geq 0$}. Then for $s\geq d$, $Q(\mu)_s$ has a quadrature rule on $\mathcal C$ with at most $\left\lceil \frac{ds-1}{2}-\frac{d(d-3)}{4} \right\rceil+1$ nodes.
 \end{prop}
 \begin{proof}

The curve $\mathcal C$ is parametrized by $(t,t^d)$.  
Therefore, a quadrature rule supported on $\mathcal C$ with nodes $(t_i,t_i^d)$ satisfies
\[
\sum_{i=1}^N \omega_i x_i^a y_i^b = \sum_{i=1}^N \omega_i t_i^{a+bd},
\qquad (a+b \le s).
\]
We formulate the following optimization problem for the nodes $t_i \ge 0$ and weights $\omega_i > 0$:
\[
\begin{cases}
\min \sum_{i=1}^N t_i^2,\\[2pt]
\sum_{i=1}^N \omega_i t_i^{a+bd} = m_{a,b}, \qquad a+b \le s.
\end{cases}
\]
Here we note that  existence of such a rule follows from the hypothesis.

Again, we observe the Lagrange conditions: 
The Lagrangian is  
\[
\mathcal L(t_i,\omega_i,\lambda_{a,b})
= \sum_{i=1}^N t_i^2
+ \sum_{a+b \le s} \lambda_{a,b}\Bigl(\sum_i \omega_i t_i^{a+bd} - m_{a,b}\Bigr).
\]
Stationarity yields
\begin{align*}
0=&\frac{\partial \mathcal L}{\partial \lambda_{a,b}} = \sum_i \omega_i t_i^{a+bd}-m_{a,b},\\
0=&\frac{\partial \mathcal L}{\partial \omega_i}\phantom{a} = \sum_{a+b\le s} \lambda_{a,b}t_i^{a+bd},\\
0=&\frac{\partial \mathcal L}{\partial t_i}\phantom{a} = \omega_i\sum_{a+b\le s}(a+bd)\lambda_{a,b}t_i^{a+bd-1}+2t_i.
\end{align*}

Let  
\[
H(t) = \sum_{a+b\le s} \lambda_{a,b}\, t^{a+bd}.
\]
Then the second condition gives $H(t_i)=0$ and the third becomes
\[
H'(t_i) = -\frac{2}{\omega_i} t_i \le 0.
\]
Thus, among the positive roots of $H$, only those with non-positive derivative are candidates for local minima; at most half of the roots can satisfy this sign condition.

We again aim to bound the number of roots via dimension arguments:  Observe that the parametrization induces a linear map
\[
\psi:\R[x,y]_{\le s} \longrightarrow \R[t]_{\le sd},
\qquad p(x,y)\longmapsto p(t,t^d).
\]
Its kernel is the principal ideal $(y-x^d)$, so we have a factorization  
\[
\overline{\psi}:(\R[x,y]/(y-x^d))_{\le s} \hookrightarrow \R[t]_{\le sd}
\]
whose image is spanned by the monomials $t^{a+bd}$ with $a+b \le s$.  

The dimension formula
\[
\dim \R[x,y]_{\le s} = {s+2\choose 2}
\]
and
\[
\dim (y-x^d)\cdot \R[x,y]_{\le s-d} = {s-d+2\choose 2}
\]
gives
\[
\dim \mathrm{im}(\psi)
= {s+2\choose 2}-{s-d+2\choose 2}
= ds - \frac{d(d-3)}{2}.
\]
Hence, any $H(t)$ constructed as above has at most  
\[
ds - \frac{d(d-3)}{2} - 1
\]
positive roots by Descartes' rule of signs.
       
Among these roots, at most half satisfy $H'(t_i)\le 0$.  
Thus, the number of local minima is at most
\[
\left\lceil \frac{ds - \frac{d(d-3)}{2} - 1}{2}\right\rceil
= \left\lceil \frac{ds-1}{2}-\frac{d(d-3)}{4}\right\rceil.
\]
Including the boundary point $t=0$ from the closed interval $[0,\infty)$ adds one more node, giving
\[
N \;\le\; \left\lceil \frac{ds-1}{2}-\frac{d(d-3)}{4} \right\rceil + 1.    \]

 \end{proof}
{
\begin{remark}
 Our proof and the connection to Hilbert series of ideals relates to the work of di Dio and Kummer \cite{diDioKummer2021}. Note however, that their results only consider the Carathéodory number of the moment cone, i.e., their results do not control possible points at {infinity} and therefore does not yield actual quadrature rules.     
\end{remark}
 }
\begin{remark}
Zalar~\cite{zalar} gives an upper bound of 
\[
N \leq d s - \left\lceil \frac{d}{2} \right\rceil
\]
for quadrature rules $Q(\mu)_{2s-1}$ supported on curves of the form $y = q(x)$ with $\deg q = d$. 
Our bounds in Theorems~\ref{thm:quadraturerational} and~\ref{thm:quadraturerationaleven} for arbitrary rational curves in $\R^n$ are, for such curves in $\R^2$, larger by at most one node. 

However, for specific curves like $y = x^d$, when the existence of a quadrature rule with positive $x$'s is guaranteed, the bound in Proposition~\ref{prop:ttd} becomes asymptotically sharper as $d \to \infty$. For example:
\begin{table}[h!]
$$
\begin{array}{c|c|c|c}
(2s-1,d) & \text{Zalar~\cite{zalar}} & \text{Theorem~\ref{thm:quadraturerational}} & \text{Proposition~\ref{prop:ttd}} \\ \hline
(3,3) & 4 & 5&  5 \\
(9,9) & 40 & 41& 28
\end{array}
$$
\caption{Upper bound on nodes for a quadrature rule of strength $2s-1$ on $y=x^d$.}
\end{table}

This shows that while Zalar's bound is slightly better for small degrees, Proposition~\ref{prop:ttd} yields significantly improved bounds for large $d$.
\end{remark}

\paragraph{Conclusion and Open Questions.}
In this article, we derived explicit bounds for quadrature rules supported on algebraic and, in particular, rational curves, allowing singularities, multiple places at infinity, and higher-dimensional embeddings. Our approach combined optimization methods with dimension arguments and root-counting techniques, yielding improved estimates over classical Carathéodory-type bounds and recovering Gaussian type optimality in the univariate case.  

Several natural questions remain open:  Firstly, while our estimates are tight in some cases (e.g., lines or certain low-degree curves), it is unclear whether they are optimal for all degrees and strengths. In particular, it seems plausible that the additional counting for places at infinity can be significantly reduced, possibly to a constant independent of $t$, or even to a single point.  Secondly, it would seem hopeful that extending the present approach to algebraic surfaces or general real varieties of dimension $\ge 2$ may reveal new phenomena. This will likely require refined tools from real algebraic geometry and moment theory to control both local and global contributions to node counts. Thirdly, it seems natural to ask if the set of quadrature rules of a given strength on a fixed curve $\mathcal{C}$ can be represented similarly to the Gaussian case, as was studied in \cite{AHL}.  

Finally, our bounds depend only on the degree and parametrization data. Incorporating finer invariants, such as singularity structure, genus, or real topology, might lead to sharper or even exact formulas for minimal node counts in broad generality.

These questions point toward an interesting  interplay between real algebraic geometry, optimization, and numerical analysis in the study of quadrature rules on algebraic sets.

\paragraph{Acknowledgements.}
The authors thank Greg Blekherman and Markus Schweighofer for insightful discussions and valuable suggestions. We are particularly grateful to Aljaž Zalar for pointing out an error in a previous version of Proposition~\ref{prop:ttd}.  We also thank the referees for a careful reading and for their useful suggestions. Financial support from the Tromsø Research Foundation through grant 17MATCR is gratefully acknowledged.

\printbibliography

\end{document}